\documentclass[10pt]{amsart}
\usepackage{amsmath,amsthm,amssymb}
\usepackage{latexsym}
\usepackage{eucal}
\usepackage{bbm}

\usepackage{fullpage}
\usepackage{verbatim, xcolor}
\usepackage{hyperref}

\usepackage{fullpage} 

\newtheorem{theorem}{Theorem}[section]
\newtheorem{lemma}{Lemma}[section]

\newcommand{\sd}{\bigtriangleup}

\theoremstyle{definition}

\def\ge{\geqslant}
\def\~{\widetilde}

\begin{document}
 \title{  On the estimate of 
 distance traveled\\ by a particle   in
 a disk-like vortex patch  
 }
\author{Kyudong Choi   }
\address{
\begin{flushleft}
\vspace{1cm} 
kchoi@unist.ac.kr\\ \vspace{0.1cm}
Ulsan National Institute of Science and Technology  \\
Department of Mathematical Sciences \\
UNIST-gil 50, Ulsan, 44919,
Republic of Korea\\
\end{flushleft}
}

\renewcommand{\thefootnote}{\fnsymbol{footnote}}
\renewcommand{\thefootnote}{\arabic{footnote}}

\begin{abstract}
We consider the incompressible two-dimensional  Euler equation in the   plane 
 in the case when its initial vorticity is the characteristic function of a bounded open set.  We show  that the travel distance  grows linearly  for most of fluid particles initially placed on  the set when the area of the symmetric difference between the set and a disk is small enough.
\end{abstract}\vspace{1cm}
\maketitle
\large

{\Large \section{Introduction}}
\noindent We consider the  incompressible 2D Euler equation in vorticity form in the whole plane:
 \begin{equation}\begin{split}\label{main_eq}
\partial_t\theta
 + u \cdot \nabla \theta
&= 
0\quad\mbox{for }  x\in\mathbb{R}^2\mbox{ and  for }  t> 0,\\
 \theta|_{t=0}&=\theta_0 \quad\mbox{for }  x\in\mathbb{R}^2
\end{split}\end{equation} where the Biot-Savart law is given by
$u=(K*\theta)$ with
 \begin{equation} \label{K}
K(x):=\frac{1}{2\pi}\frac{x^\perp}{|x|^2}=\frac{1}{2\pi}\Big(-\frac{x_2}{|x|^2},\frac{x_1}{|x|^2}\Big).\end{equation}
When $\theta_0$ lies on $L^1\cap L^\infty$, the existence and uniqueness of a global-in-time weak solution is due to Yudovich \cite{yu}.\\ 

\noindent In this paper, we are interested in estimating the  distance traveled by a fluid particle.
More precisely, we consider the case when  the initial data $\theta_0$ is the characteristic function ${\mathbbm{1}}_{\Omega_0}$  of a   bounded open set $\Omega_0$ in $\mathbb{R}^2$.
Then the corresponding solution $\theta$ is given by Yudovich theory, and it has the form of $\theta(t)={\mathbbm{1}}_{\Omega_t}$ where $\Omega_t$ is defined by
$\Omega_t=\{
\phi_x(t) \in\mathbb{R}^2\,|\, x\in\Omega_0
\}$ and $\phi_x(\cdot)$ is the particle trajectory of  the particle whose initial position is at $x\in\mathbb{R}^2$, which can be obtained by solving the following system of the ordinary differential equations:
$$\frac{d}{dt}\phi_x(t)=u(t,\phi_x(t))\quad\mbox{for }t>0\quad\mbox{and } \quad \phi_x(0)=x.$$
This is well-defined since the velocity field $u$ allows a log-Lipschitz estimate (\textit{e.g.} see the modern texts \cite{mb}, \cite{MaPu}).
For $t>0$ and $x\in\mathbb{R}^2$, we say that $d_x(t)$ is   the   distance traveled by  the particle, whose initial position is at $x$, up to time $t$. More precisely, we define $d_x(t)$ by
$$d_x(t):= \int_0^t |u(s,\phi_x(s))|ds.$$ 

\noindent For instance, when $\Omega_0$ is the unit disk $D$ centered at the origin, we  easily compute  (see Subsection \ref{subsec_example})
$$d_x(t)=\frac{|x|}{2}t \quad\mbox{for  any } x \in D
\mbox{ and for any } t>0.$$
 For a  general bounded open set $\Omega_0$, 
we have, at least, a trivial linear upper bound  
 $$ {d_x(t)}\leq Ct\quad\mbox{for  any } x \mbox{ in the plane  and for any } t>0$$ 
where the above constant $C$ depends only on the Lebesgue measure $|\Omega_0|$ of the set $\Omega_0$. Indeed,   the conservation    of the total mass 
 $|{\Omega_t}| $ in time due to the incompressibility of the fluid gives a uniform bound for  $u$ (\textit{e.g.} see estimate   \eqref{est_bounded_velocity} with the substitution $f={\mathbbm{1}}_{\Omega_0}$). \\

\noindent The main result of this paper says that,      for most of particles  placed on  the initial vortex patch in the beginning, the travel distance actually grows linearly in time when the initial  patch is disk-like in the sense that the measure of the initial symmetric difference $$\Omega_0\bigtriangleup B_r:=(\Omega_0\backslash B_r)\cup(B_r\backslash\Omega_0)$$ is small enough. Here we denote $B_r:=\{x\in\mathbb{R}^2\,|\,|x|<r\}$ for  $r>0$.  
Without loss of generality, we will compare the initial patch $\Omega_0$ only with the unit disk $$D:=B_1$$  thanks to  the scaling  of the Euler equation.\\

  \begin{theorem}\label{thm_infinite}
  For any $R>0$, 
  there exist  constants $\delta_0>0$, $C>0$ and $c>0$  such that if
$\Omega_0\subset B_{R}$  and if $$|D\sd\Omega_0  |<\delta_0,$$ then the 
 solution of \eqref{main_eq} for the initial data $\theta_0={\mathbbm{1}}_{\Omega_0}$ satisfies the following two properties:\\ \indent $(I)$ For any $\delta>0$ satisfying $|D\sd\Omega_0 |\leq \delta\leq \delta_0$, there exists a set 
  $H_\delta\subset \Omega_0$ such that $$|H_\delta|\geq |\Omega_0|-C\delta^{1/4}$$ and 
\begin{equation}\label{est_main_infinite}
\limsup_{t\to\infty} \frac{d_x(t)}{t}\geq c{\delta^{1/2}}
\end{equation}
 for every $x\in H_\delta$.  \\
\indent  $(II)$ For any $\delta>0$ satisfying $|D\sd\Omega_0 |\leq\delta\leq \delta_0$ 
 and
for each  $T>0$, there exists
 a set $H_{\delta, T}\subset \Omega_0$ such that
$$|H_{\delta, T}|\geq |\Omega_0|-C\delta^{1/4}
$$ and 
\begin{equation}\label{thm_finite_estimate}\frac{d_x(T)}{T}\geq   c\delta^{1/4}  \end{equation} 
 for every $x\in H_{\delta, T}$.
 \end{theorem}

  



 
\noindent 
In the proof, we  use  the $L^1$-stability result  in  \cite{sv}.
Indeed, 
when the initial patch is disk-like,   we can show that
 for each fixed time moment, most of particles, which are initially placed on the initial set $\Omega_0$, should be detected 
in the  annulus region $\{\epsilon\leq |x|<1\}$ 
 thanks to the incompressibility of the fluid. Here $\epsilon>0$ will be taken small depending on $\delta>0$. This gives Lemma \ref{stability_consequence}. 
Then we prove that 
most of particles spend most of their life time in the  annulus     (\textit{e.g.} see \eqref{most_time}). 
Since 
the speed induced by the exact disk patch is non-trivial within the region, we get the conclusion by  using the stability result again.  \\

\noindent For the vortex patch problem, there are many other interesting results including persistence of boundary regularity   \cite{Che}, \cite{BC} (or see the textbook \cite{Che_book} and references therein,  also  see \cite{KRYZ} for a blow-up result in a modified SQG patch equation), existence   of rotating patches \cite{DeZa}, \cite{denisov_v_state}, \cite{Vstate}, \cite{HmMaVe}, which are so-called ``V-states'', 
 and stability of circular patches   \cite{wp}, \cite{Dri}, \cite{sv} (also see \cite{bd} for rectangular patches in a 2D infinite cylinder). 
Confinement of patch evolution, or support of positive vorticity evolution, is  also interesting (\textit{e.g. }see \cite{marchioro}, \cite{isg}, \cite{Serfati}, \cite{CD}, \cite{iln}).\\ 

\noindent  Before proving our theorem, we present the simplest example for travel distance.
\subsection{When $\Omega_0=D$}\label{subsec_example} \ \\
If we consider the initial data $\theta_0={\mathbbm{1}}_D$, then the radial symmetry of the data implies that corresponding solution is stationary. The velocity induced by the stationary vorticity $\theta_t={\mathbbm{1}}_D$ is 
\begin{equation}\label{no_rad}
u(x)=\begin{cases} &\frac{ x^{\perp}}{2} \quad\mbox{if } x\in D,\\
&\frac{x^{\perp}}{2|x|^2} \quad\mbox{otherwise }  
\end{cases} \end{equation}   (e.g. see \cite{mb}). 
If we decompose into its radial part and tangential part:
$u =u_{rad}\frac{x}{|x|}+u_{tan}\frac{x^\perp}{|x|},$ then  
we get $u_{rad}=0$
 and
$$u_{tan}(x)=\begin{cases} &\frac{|x|}{2} \quad\mbox{if } x\in D,\\
&\frac{1}{2|x|} \quad\mbox{otherwise. }  
\end{cases} $$
It says that each particle in the disk just rotates in a constant angular velocity. Thus we simply have, for any $x\in D$ and for any $t> 0$, $$\frac{d_x(t)}{t}=\frac{|x|}{2}$$
and $ d_x(\infty) =\infty $ unless $x=0$.
\section{Proof}
\subsection{$L^1$-stability of a disk patch} \ \\
The paper \cite{sv} showed the following $L^1$-stability of a circular vortex patch:
\begin{lemma}[Theorem 3 in \cite{sv}]
For any bounded open set $\Omega_0\subset \mathbb{R}^2$ and for any $r>0$, we have 
\begin{equation}\label{sv_result_original}
\|{\mathbbm{1}}_{\Omega_t}-{\mathbbm{1}}_{B_r}\|_{L^1}^2\leq4\pi\cdot \sup_{x\in\Omega_0\bigtriangleup B_r}||x|^2-r^2|\cdot
\|{\mathbbm{1}}_{\Omega_0}-{\mathbbm{1}}_{B_r}\|_{L^1} \quad\mbox{for any }t> 0.
\end{equation}
\end{lemma}
\noindent In \cite{sv}, the authors used conservation of mass, momentum and moment of inertia to prove the above result.\\

\noindent We note   $\|{\mathbbm{1}}_{A}-{\mathbbm{1}}_B\|_{L^1}=\int_{A\bigtriangleup  B}1 dx=
|A\bigtriangleup  B|$ for any bounded open sets $A$ and $B$.   In this paper, we will use the result \eqref{sv_result_original} of the above lemma
 with the choice $r=1$:
\begin{equation}\label{sv_result}
|\Omega_t \bigtriangleup D|^2 \leq4\pi\cdot \sup_{x\in\Omega_0\bigtriangleup D}||x|^2-1|\cdot |\Omega_0 \bigtriangleup D|  \quad\mbox{for any }t> 0
\end{equation} in the following way:\\
\indent Let 
$R>0$.
Assume $\Omega_0\subset B_R$ with
$\alpha_0:=|\Omega_0 \bigtriangleup D|>0. 
$
Consider the initial data 
$\theta_0={\mathbbm{1}}_{\Omega_0}$ with its solution   $\theta(t)={\mathbbm{1}}_{\Omega_t}$.
Since $\sup_{\Omega_0\bigtriangleup D}||x|^2-1|\leq (R^2+1)$,
the stability result \eqref{sv_result} implies 
\begin{equation}\label{est_sta}
|\Omega_t \bigtriangleup D|\leq 2\sqrt{\pi(R^2+1)\alpha_0}.
\end{equation}

\noindent On the other hand, we observe that there exists a constant $C_0>0$ such that
 \begin{equation}\label{est_bounded_velocity}
 \|K*f\|_{L^\infty}\leq C_0\|f\|^{1/2}_{L^1}\|f\|^{1/2}_{L^\infty}
 \end{equation} for any $f\in (L^1\cap L^\infty)(\mathbb{R}^2).$ This can be proved by using the obvious estimate $|K(x)|\leq \frac{C}{|x|}$ of the Biot-Savart kernel \eqref{K} (or see Lemma 2.1. in \cite{isg}).
 In our setting,  we have 
$$\|K*({\mathbbm{1}}_{\Omega_t}-{\mathbbm{1}}_{D})\|_{L^\infty}\leq C_0\|{\mathbbm{1}}_{\Omega_t}-{\mathbbm{1}}_{D}\|^{1/2}_{L^1}\|{\mathbbm{1}}_{\Omega_t}-{\mathbbm{1}}_{D}\|^{1/2}_{L^\infty}\leq C_1\alpha_0^{1/4}\quad \mbox{where }$$  \begin{equation}\label{def_c1}
C_1=C_1(R):=C_0\sqrt{2\sqrt{\pi(R^2+1)}}>0.
\end{equation} 

\noindent Since the vector field $(K*{\mathbbm{1}}_{D})$ has no radial component by \eqref{no_rad} and the velocity $u(t)$ from $\theta(t)={\mathbbm{1}}_{\Omega_t} $ can be decomposed:  $ u(t)=K*[\theta(t)]=K*{\mathbbm{1}}_{\Omega_t} 
=K*{\mathbbm{1}}_{D}+K*({\mathbbm{1}}_{\Omega_t}-{\mathbbm{1}}_{D}),$  we obtain \begin{equation*}
|u_{rad}(t,x)|\leq C_1\alpha_0^{1/4}
\end{equation*} for any $x\in\mathbb{R}^2$ and for any $t>0$. 
 For the tangential component, 
we have, 
 for    $x\in D$ 
and for   $t> 0$,   
\begin{equation}\label{est_speed_lower}
|u_{tan}(t,x)|\geq \frac{|x|}{2}-C_1\alpha_0^{1/4}.
\end{equation} 
In this setting, we can prove the following lemma:
\begin{lemma}\label{stability_consequence}
Let $0<\epsilon<1$. Suppose 
$\Omega_0\subset B_R$ and $2\sqrt{\pi(R^2+1)\cdot |\Omega_0 \bigtriangleup D|}\leq \epsilon^2\pi.$  Then for each finite $T>0$, we have
\begin{equation*}
\int_{\Omega_0}\Big(\frac{1}{T}\int_0^T{\mathbbm{1}}_{B_\epsilon\cup B_1^C}(\phi_x(t))dt \Big)dx
\leq 2\epsilon^2\pi.
\end{equation*}
\end{lemma}
\begin{proof}
We define  
\begin{equation*}
F(t,x)={\mathbbm{1}}_{B_\epsilon\cup B_1^C}(\phi_x(t))=\begin{cases} 1\quad\mbox{ if either }|\phi_x(t)|< \epsilon  \mbox{ or }  |\phi_x(t)|\geq 1 ,\\
0 \quad \mbox{otherwise}.
\end{cases}\end{equation*}
For each $t>0$, since
the velocity $u$ is divergence-free, the flow map $X(t,x):=\phi_x(t)$ is area preserving so that  we have 
\begin{equation*}\begin{split}\int_{\Omega_0}F(t,x)dx
&=\int_{\Omega_0}{\mathbbm{1}}_{B_\epsilon\cup B_1^C}(\phi_x(t))dx=\int_{\Omega_t\cap B_\epsilon}1 dx+\int_{\Omega_t\cap B_1^C}1 dx\\
&\leq 
\int_{ B_\epsilon}1 dx+\int_{\Omega_t\bigtriangleup D}1 dx\leq |B_\epsilon| +|\Omega_t\bigtriangleup D|
\\
&\leq \epsilon^2\pi +2\sqrt{\pi(R^2+1)\alpha_0}\leq 2\epsilon^2\pi
\end{split}\end{equation*} by \eqref{est_sta} and by the assumption in this lemma.  
Let $T>0$. Then, by Fubini, we get 
$
\int_{\Omega_0} { \int_0^TF(t,x)dt} dx=  \int_0^T\int_{\Omega_0}F(t,x)dxdt
\leq 2\epsilon^2\pi T. 
$

\end{proof}

\subsection{Proof of $(I)$ in Theorem \ref{thm_infinite}: Travel distance in infinite time}  
  
   \begin{proof}[Proof of $(I)$ in Theorem \ref{thm_infinite}]

Let $R>0$. 
   First, we take 
   $\delta_0:=\Big(\min\{\frac{1}{8C_1},\frac{1}{2}\sqrt{\frac{\pi}{2\sqrt{\pi(R^2+1)}}}\}\Big)^4>0$
   where $C_1=C_1(R)>0$ is defined in \eqref{def_c1}.
Assume $\Omega_0\subset B_{R}$ and $\alpha_0:=|\Omega_0 \bigtriangleup D|<\delta_0$ and take any $\delta>0$ satisfying $\alpha_0\leq\delta\leq\delta_0$. Define  $\epsilon=\epsilon(\delta):=
   C_2\cdot \delta^{1/4}$
   where $C_2:=    \max\{4C_1,\sqrt{\frac{2\sqrt{\pi(R^2+1) }}{\pi}}\}>0$. Clearly, we have $0<\epsilon\leq 1/2$,  
   \begin{equation}\label{cond_1}
    C_1\alpha_0^{1/4}
   \leq C_1\delta^{1/4}
    \leq \frac{\epsilon}{4}
   \end{equation} and  
      \begin{equation}\label{cond_2}
  2\sqrt{\pi(R^2+1)\alpha_0}  \leq 2\sqrt{\pi(R^2+1)\delta}  \leq \epsilon^2\pi.
   \end{equation}
   
 \noindent  For each $n\in\mathbb{N}$ and for $x\in\Omega_0$, we define $\psi_n(x):=\frac{1}{n}\int_0^n{\mathbbm{1}}_{D\backslash{B_\epsilon}}(\phi_x(t))dt.$ We observe $0\leq \psi_n\leq 1$. By Lemma \ref{stability_consequence} due to \eqref{cond_2}, we   have
\begin{equation*} 
 \int_{\Omega_0}\psi_n dx=
\int_{\Omega_0}\Big(\frac{1}{n}\int_0^n{\mathbbm{1}}_{D\backslash{B_\epsilon}}(\phi_x(t))dt \Big)dx
=\int_{\Omega_0}\frac{1}{n}\int_0^n\Big(1-{\mathbbm{1}}_{ {B_\epsilon}\cup B_1^C}(\phi_x(t))\Big)dt dx\geq |\Omega_0|-2\epsilon^2\pi.
 \end{equation*}
 We put $f_n(x)=\sup_{m\geq n}\psi_m(x)$ for each $n$. Then  we notice that $0\leq f_n\leq 1$ and  $\{f_n\}_{n=1}^\infty$ decays pointwise to some function   $ f\geq 0$.   By Dominated Convergence Theorem (\textit{e.g.} see \cite{folland}), we have 
$\int_{\Omega_0} fdx=\lim_{n\to\infty} \int_{\Omega_0} f_ndx$. From
  $\int_{\Omega_0} f_ndx\ge \int_{\Omega_0} \psi_ndx\geq|\Omega_0|-2\epsilon^2\pi$,  we get  $$\int_{\Omega_0} fdx\geq|\Omega_0|-2\epsilon^2\pi.$$

\noindent By putting $H:=\{x\in\Omega_0\,|\,f(x)>\epsilon \}$,  
 we have 
$$|\Omega_0|-2\epsilon^2\pi \leq \int_{\Omega_0} f dx=\int_{\Omega_0\cap \{f\leq \epsilon \}}f dx+\int_{\Omega_0\cap \{f> \epsilon \}}f dx
\leq \epsilon |\Omega_0|+|H|.
$$ Thus  we get $$|H|\geq  |\Omega_0|-\epsilon( |\Omega_0|+2 \pi\epsilon) \geq  |\Omega_0|-\epsilon( |D|+|D\sd \Omega_0|+2 \pi)\geq  |\Omega_0|-\epsilon( \pi+\delta_0+2 \pi)\geq  |\Omega_0|-C\delta^{1/4}$$ for some $C>0$  by $\epsilon\sim \delta^{1/4}$.  
We observe 
 \[
 f(x)=
 \limsup_{n\to\infty}\frac{1}{n}\int_0^n{\mathbbm{1}}_{D\backslash B_\epsilon}(\phi_x(t))dt.
 \]  
 Thus for each $x\in H$, there is an increasing sequence $\{n_k(x)\}_{k=1}^\infty$ in $\mathbb{N}$   such that
 $\lim_{k\to\infty}n_k(x)=\infty$ and
 \[
 \frac{1}{n_k(x)}|\{t\in[0,n_k(x)]\,|\, \epsilon\leq |\phi_x(t)|<1   \}|=\frac{1}{n_k(x)}\int_0^{n_k(x)}{\mathbbm{1}}_{D\backslash B_\epsilon}(\phi_x(t))dt\geq {\epsilon}.
 \]

\noindent   On the other hand,
  whenever $(t,x)$ satisfies
$\epsilon\leq |\phi_x(t)|<1 $,
we have 
\begin{equation}\label{est_speed_lower_result}
|u (t,\phi_x(t))|\geq |u_{tan}(t,\phi_x(t))|\geq \frac{|\phi_x(t)|}{2}-C_1\alpha_0^{1/4}\geq \frac{\epsilon}{2}-C_1\delta^{1/4}\geq \frac{\epsilon}{4} \end{equation} 
by \eqref{est_speed_lower} and \eqref{cond_1}. 
 Thus, for any $x\in H$ and for any $k\in\mathbb{N}$, we have $$d_x(n_k(x))
 =\int_0^{n_k(x)}|u(s,\phi_x(s))|ds\geq (n_k(x)\cdot{\epsilon})\cdot\frac{\epsilon}{4}$$ which implies
  $
 \limsup_{t\to\infty} \frac{d_x(t)}{t}\geq \frac{\epsilon^2}{4}
$
  for each $x\in H$. Since $\epsilon\sim \delta^{1/4}$, we get \eqref{est_main_infinite}.
 
 \end{proof}

\subsection{Proof of $(II)$ in Theorem \ref{thm_infinite}: Travel distance in finite time}  
  
  \begin{proof}[Proof of $(II)$ in Theorem \ref{thm_infinite}]
  We take the same $\delta_0, \alpha_0,\delta,\epsilon$ as in the proof of  Theorem \ref{thm_infinite} so that we have 
  $0<\epsilon\leq 1/2$,  
\eqref{cond_1} and \eqref{cond_2}.

\noindent   Let $T>0$. By Lemma \ref{stability_consequence} due to \eqref{cond_2}, we have 
$$
\int_{\Omega_0}\underbrace{\Big(\frac{1}{T}\int_0^T{\mathbbm{1}}_{B_\epsilon\cup B_1^C}(\phi_x(t))dt \Big)}_{=:G_T(x)}dx
\leq 2\epsilon^2\pi.
$$
 
\noindent By Chebyshev (\textit{e.g.} see \cite{folland}), we get 
$|\{x\in\Omega_0\,|\, G_T(x)\geq \epsilon\}|\leq {2\epsilon\pi  }.$
 
\noindent We put $H_T:=\{x\in\Omega_0\,|\, G_T(x)<\epsilon  \}$. Then we observe 
$|H_T|\geq |\Omega_0|- {2\epsilon \pi}\geq |\Omega_0|- C\delta^{1/4}
$ for some $C>0$ by $\epsilon\sim \delta^{1/4}$.
We also observe that for any $x\in H_T$,
$$|\{t\in[0,T]\,|\, |\phi_x(t)|<\epsilon \mbox{ or } |\phi_x(t)|\geq 1 \}|=\int_0^T
{\mathbbm{1}}_{B_\epsilon\cap B_1^C}(\phi_x(t))dt=T\cdot G_T(x)<\epsilon T.$$
Thus for each $x\in H_T$, we get 
\begin{equation}\label{most_time}
|\{t\in[0,T]\,|\, \epsilon\leq |\phi_x(t)|<1   \}|\geq (1-\epsilon) T.
\end{equation}


\noindent By \eqref{est_speed_lower} and \eqref{cond_1}, 
if
$\epsilon\leq |\phi_x(t)|<1 $,
then   we get 
$|u (t,\phi_x(t))|
\geq \frac{\epsilon}{4}$ 
as in \eqref{est_speed_lower_result}.
Thus, for any $x\in H_T$,  we get 
$
d_x(T)\geq (1-\epsilon) T\cdot \frac{\epsilon}{4}\geq T\cdot \frac{\epsilon}{8} 
$
because of   $\epsilon\leq 1/2$. Since $\epsilon\sim \delta^{1/4}$, we obtain \eqref{thm_finite_estimate}.


\end{proof}

{\Large \section*{acknowledgement}}

\noindent KC was supported by   the National Research Foundation of Korea (NRF-2018R1D1A1B07043065) and by the POSCO Science Fellowship of POSCO TJ Park Foundation. We thank Prof. S. Denisov  for many helpful discussions.\\

\end{document}